\documentclass[
11pt,  reqno]{amsart}
\usepackage[margin=1.2in,marginparwidth=1.5cm, marginparsep=0.5cm]{geometry}

\usepackage{amsmath,amssymb,amscd,amsthm,amsxtra, esint}

\usepackage{booktabs} 
\usepackage{microtype}
\usepackage{amssymb}
\usepackage{mathrsfs}

\usepackage{upgreek}

\usepackage{color}
\usepackage[implicit=true]{hyperref}

\usepackage{xsavebox}

\usepackage{cases}

\allowdisplaybreaks[2]

\definecolor{gr}{rgb}   {0.,   0.69,   0.23 }
\definecolor{bl}{rgb}   {0.,   0.5,   1. }
\definecolor{mg}{rgb}   {0.85,  0.,    0.85}
\definecolor{yl}{rgb}   {0.8,  0.7,   0.}
\definecolor{or}{rgb}  {0.7,0.2,0.2}

\usepackage{tikz}
\usepackage{marginnote}
\usepackage{scalerel} 

\usetikzlibrary{shapes.misc}
\usetikzlibrary{shapes.symbols}
\usetikzlibrary{decorations}
\usetikzlibrary{decorations.markings}

\tikzset{
	dot/.style={circle,fill=black,draw=black,inner sep=0pt,minimum size=0.5mm},
	>=stealth,
	}

\tikzset{
	dot2/.style={circle,fill=black,draw=black,inner sep=0pt,minimum size=0.2mm},
	>=stealth,
	}

\tikzset{
	ddot/.style={circle,fill=black,draw=black,inner sep=0pt,minimum size=0.8mm},
	>=stealth,
	}

\tikzset{decision/.style={ 
        draw,
        diamond,
        aspect=1.5
    }}

\tikzset{dia2/.style
={diamond,fill=white,draw=black,inner sep=0pt,minimum size=1mm},
	>=stealth,
	}

\tikzset{dia/.style
={star,fill=black,draw=black,inner sep=0pt,minimum size=1mm},
	>=stealth,
	}

\tikzset{dia/.style
={diamond,fill=black,draw=black,inner sep=0pt,minimum size=1.3mm},
	>=stealth,
	}

\makeatletter
\def\DeclareSymbol#1#2#3{\xsavebox{#1}{\tikz[baseline=#2,scale=0.15]{#3}}}
\def\<#1>{\xusebox{#1}}
\makeatother

\newsavebox{\peA}
\newsavebox{\pneA}
\newsavebox{\plA}
\newsavebox{\pgA}
\newsavebox{\pleA}
\newsavebox{\pgeA}
\newsavebox{\pezA}

\savebox{\peA}{\tikz \draw (0,0) node[shape=circle,draw,inner sep=0pt,minimum size=8.5pt] {\scriptsize  $=$};}
\savebox{\pneA}{\tikz \draw (0,0) node[shape=circle,draw,inner sep=0pt,minimum size=8.5pt] {\footnotesize $\neq$};}
\savebox{\plA}{\tikz \draw (0,0) node[shape=circle,draw,inner sep=0pt,minimum size=8.5pt] {\scriptsize $<$};}
\savebox{\pgA}{\tikz \draw (0,0) node[shape=circle,draw,inner sep=0pt,minimum size=8.5pt] {\scriptsize $>$};}
\savebox{\pleA}{\tikz \draw (0,0) node[shape=circle,draw,inner sep=0pt,minimum size=8.5pt] {\scriptsize $\leqslant$};}
\savebox{\pgeA}{\tikz \draw (0,0) node[shape=circle,draw,inner sep=0pt,minimum size=8.5pt] {\scriptsize $\geqslant$};}
\savebox{\pezA}{\tikz \draw (0,0) node[shape=circle,draw,
fill=white, 
inner sep=0pt,minimum size=8.5pt]{} ;}

\def \peB{\mathchoice
{\scalebox{.7}{{\usebox{\peA}}}}
{\scalebox{.7}{{\usebox{\peA}}}}
{\scalebox{.7}{{\usebox{\peA}}}}
{}
}

\def \pezB{\mathchoice
{\scalebox{.7}{{\usebox{\pezA}}}}
{\scalebox{.7}{{\usebox{\pezA}}}}
{\scalebox{.7}{{\usebox{\pezA}}}}
{}
}

\newcommand{\pe}{\mathbin{{\peB}}}

\newcommand{\pez}{\mathbin{{\pezB}}}

\usetikzlibrary{decorations.pathmorphing, patterns,shapes}

\DeclareSymbol{dot}{-2.7}{\draw (0,0) node[dot]{};}

\DeclareSymbol{1}{0}{\draw[white] (-.4,0) -- (.4,0); \draw (0,0)  -- (0,1.2) node[dot] {};}
\DeclareSymbol{2}{0}{\draw (-0.5,1.2) node[dot] {} -- (0,0) -- (0.5,1.2) node[dot] {};}
\DeclareSymbol{3}{0}{\draw (0,0) -- (0,1.2) node[dot] {}; \draw (-.7,1) node[dot] {} -- (0,0) -- (.7,1) node[dot] {};}
\DeclareSymbol{30}{-3}{\draw (0,0) -- (0,-1); \draw (0,0) -- (0,1.2) node[dot] {}; \draw (-.7,1) node[dot] {} -- (0,0) -- (.7,1) node[dot] {};}
\DeclareSymbol{31}{-3}{\draw (0,0) -- (0,-1) -- (1,0) node[dot] {}; \draw (0,0) -- (0,1.2) node[dot] {}; \draw (-.7,1) node[dot] {} -- (0,0) -- (.7,1) node[dot] {};}
\DeclareSymbol{32}{-3}{\draw (0,0) -- (0,-1) -- (1,0) node[dot] {}; \draw (0,0) -- (0,-1) -- (-1,0) node[dot] {}; \draw (0,0) -- (0,1.2) node[dot] {}; \draw (-.7,1) node[dot] {} -- (0,0) -- (.7,1) node[dot] {};}
\DeclareSymbol{20}{-3}{\draw (0,0) -- (0,-1);\draw (-.7,1) node[dot] {} -- (0,0) -- (.7,1) node[dot] {};}
\DeclareSymbol{22}{-3}{\draw (0,0.3) -- (0,-1) -- (1,0) node[dot] {}; \draw (0,0.3) -- (0,-1) -- (-1,0) node[dot] {};\draw (-.7,1) node[dot] {} -- (0,0.3) -- (.7,1) node[dot] {};}
\DeclareSymbol{31p}{-3}{\draw (0,0) -- (0,-1) -- (1,0) node[dot] {}; \draw (0,0) -- (0,1.2) node[dot] {}; \draw (-.7,1) node[dot] {} -- (0,0) -- (.7,1) node[dot] {};
\draw (0,-1) node{\scaleobj{0.5}{\pez}};
\draw (0,-1) node{\scaleobj{0.5}{\pe}}; }
\DeclareSymbol{32p}{-3}{\draw (0,0) -- (0,-1) -- (1,0) node[dot] {}; \draw (0,0) -- (0,-1) -- (-1,0) node[dot] {}; \draw (0,0) -- (0,1.2) node[dot] {}; \draw (-.7,1) node[dot] {} -- (0,0) -- (.7,1) node[dot] {};
\draw (0,-1) node{\scaleobj{0.5}{\pez}};
\draw (0,-1) node{\scaleobj{0.5}{\pe}};}
\DeclareSymbol{22p}{-3}{\draw (0,0.3) -- (0,-1) -- (1,0) node[dot] {}; \draw (0,0.3) -- (0,-1) -- (-1,0) node[dot] {};\draw (-.7,1) node[dot] {} -- (0,0.3) -- (.7,1) node[dot] {};
\draw (0,-1) node{\scaleobj{0.5}{\pez}};
\draw (0,-1) node{\scaleobj{0.5}{\pe}};}
\DeclareSymbol{21p}{-3}{\draw (0,0.3) -- (0,-1) -- (1,0) node[dot] {};
\draw (-.7,1) node[dot] {} -- (0,0.3) -- (.7,1) node[dot] {};
\draw (0,-1) node{\scaleobj{0.5}{\pez}};
\draw (0,-1) node{\scaleobj{0.5}{\pe}};}

\DeclareSymbol{21}{-3}{\draw (0,0.3) -- (0,-1) -- (1,0) node[dot] {};
\draw (-.7,1) node[dot] {} -- (0,0.3) -- (.7,1) node[dot] {};
}

\DeclareSymbol{11p}{-3}{\draw (0,0) node[dot2] {} -- (0,-1) -- (1,0) node[dot] {};
\draw (0,0) -- (0,1.2) node[dot] {};  
\draw (0,-1) node{\scaleobj{0.5}{\pez}};
\draw (0,-1) node{\scaleobj{0.5}{\pe}}; }
\DeclareSymbol{100}{-3}
{\draw[white] (-.4,0) -- (.4,0);
\draw (0,0) node[dot2] {} -- (0,-1)  ;
\draw (0,0) -- (0,1.2) node[dot] {};}  
\def\C{\mathbb{C}}

\def\R{\mathbb{R}}

\def\r2n{{\mathbb{R}^{2n}}}

\def\N{\mathbb{N}}

\def\Z{\mathbb{Z}}

\def\supp{\operatorname{supp}}

\usetikzlibrary{mindmap,backgrounds,trees,arrows,decorations.pathmorphing,decorations.pathreplacing,positioning,shapes.geometric,shapes.misc,decorations.markings,decorations.fractals,calc,patterns}

\tikzset{>=stealth',
         cvertex/.style={circle,draw=black,inner sep=1pt,outer sep=3pt},
         vertex/.style={circle,fill=black,inner sep=1pt,outer sep=3pt},
         star/.style={circle,fill=yellow,inner sep=0.75pt,outer sep=0.75pt},
         tvertex/.style={inner sep=1pt,font=\scriptsize},
         gap/.style={inner sep=0.5pt,fill=white}}
\tikzstyle{mybox} = [draw=black, fill=blue!10, very thick,
    rectangle, rounded corners, inner sep=10pt, inner ysep=20pt]
\tikzstyle{boxtitle} =[fill=blue!50, text=white,rectangle,rounded corners]

\tikzstyle{decision} = [diamond, draw, fill=blue!20,
    text width=4.5em, text badly centered, node distance=2.5cm, inner sep=0pt]
\tikzstyle{block} = [rectangle, draw, fill=blue!20,
    text width=2.7cm, text centered, rounded corners, minimum height=3em]

\tikzstyle{line} = [draw, very thick, color=black!50, -latex']

\tikzstyle{cloud} = [draw, ellipse,fill=red!40,
node distance=2.5cm,
    minimum height=1em]

\tikzstyle{cloud2} = [draw, ellipse,fill=red!30, text=white,text width=10em, node distance=2.5cm, text centered, minimum height=4em]

\tikzstyle{cloud3} = [draw, ellipse, fill=cyan!30,
node distance=2.5cm,
    minimum height=3em,
    text width=1.7cm]

\tikzstyle{cloud4} = [draw, ellipse,fill=orange!70, node distance=2.5cm,
   text width=2.1cm,
           minimum height=3em]

\tikzstyle{cloud5} = [draw, ellipse,fill=red!20, node distance=2.5cm,
    text centered,
    text width=3.0cm,
    minimum height=3em]

\tikzstyle{cloud6} = [draw, ellipse,fill=red!20, node distance=2.5cm,
    text width=1.5cm,
    text centered,
    minimum height=3em]

\tikzset{
    position/.style args={#1:#2 from #3}{
        at=(#3.#1), anchor=#1+180, shift=(#1:#2)
    }
}

\newtheorem{theorem}{Theorem} [section]

\newtheorem{lemma}[theorem]{Lemma}
\newtheorem{proposition}[theorem]{Proposition}
\newtheorem{remark}[theorem]{Remark}
\newtheorem{example}{Example}
\newtheorem{question}[theorem]{Question}

\newtheorem{oldtheorem}{Theorem}

\newcommand{\noi}{\noindent}
\newcommand{\T}{\mathbb{T}}

\let\P= \undefined
\newcommand{\P}{\mathbf{P}}

\newcommand{\dl}{\delta}

\newcommand{\eps}{\varepsilon}

\newcommand{\ft}{\widehat}

\newcommand{\wt}{\widetilde}
\newcommand{\cj}{\overline}

\newcommand{\dt}{\partial_t}

\renewcommand{\l}{\ell}

\newcommand{\les}{\lesssim}

\newcommand{\jb}[1]
{\langle #1 \rangle}

\renewcommand{\S}{\mathcal{S}}

\newtheorem*{ackno}{Acknowledgements}

\numberwithin{equation}{section}
\numberwithin{theorem}{section}

\newcommand{\too}{\longrightarrow}

\usepackage{comment}

\newcommand{\BMO}{\textit{BMO} }
\newcommand{\CMO}{\textit{CMO} }

\newcommand{\CZ}{Calder\'on-Zygmund }

\makeatletter
\DeclareRobustCommand\widecheck[1]{{\mathpalette\@widecheck{#1}}}
\def\@widecheck#1#2{%
   \setbox\z@\hbox{\m@th$#1#2$}%
   \setbox\tw@\hbox{\m@th$#1%
      \widehat{%
         \vrule\@width\z@\@height\ht\z@
         \vrule\@height\z@\@width\wd\z@}$}%
   \dp\tw@-\ht\z@
   \@tempdima\ht\z@ \advance\@tempdima2\ht\tw@ \divide\@tempdima\thr@@
   \setbox\tw@\hbox{%
      \raise\@tempdima\hbox{\scalebox{1}[-1]{\lower\@tempdima\box\tw@}}}%
   {\ooalign{\box\tw@ \cr \box\z@}}}
\makeatother

\makeatletter
\@namedef{subjclassname@2020}{%
  \textup{2020} Mathematics Subject Classification}
\makeatother

\begin{document}
\baselineskip = 14pt

\title[Compact bilinear operators]
{Compact bilinear operators and paraproducts revisited}

\author[\'A. B\'enyi, G. Li, T. Oh, and R. H. Torres]
{\'Arp\'ad B\'enyi, Guopeng Li, Tadahiro Oh, and Rodolfo H. Torres}

\address{\'Arp\'ad B\'enyi, Department of Mathematics\\
516 High St, Western Washington University\\ Bellingham, WA 98225,
USA.}

\email{benyia@wwu.edu}

\address{
Guopeng Li, School of Mathematics\\
The University of Edinburgh\\
and The Maxwell Institute for the Mathematical Sciences\\
James Clerk Maxwell Building\\
The King's Buildings\\
Peter Guthrie Tait Road\\
Edinburgh\\
EH9 3FD\\
United Kingdom, 
and 
Department of Mathematics and Statistics\\
Beijing Institute of Technology\\
Beijing\\ China}

\email{guopeng.li@ed.ac.uk}

\address{
Tadahiro Oh, School of Mathematics\\
The University of Edinburgh\\
and The Maxwell Institute for the Mathematical Sciences\\
James Clerk Maxwell Building\\
The King's Buildings\\
Peter Guthrie Tait Road\\
Edinburgh\\
EH9 3FD\\
 United Kingdom}

\email{hiro.oh@ed.ac.uk}

\address{Rodolfo H. Torres,
Department of Mathematics\\
University of California, Riverside\\
200 University Office Building\\
Riverside, CA 92521, USA}

\email{rodolfo.h.torres@ucr.edu}



\begin{abstract}
We present a new proof of the compactness of bilinear paraproducts with \CMO symbols.
By drawing an analogy to compact linear operators,
we  first explore further properties
of compact bilinear operators  on Banach spaces
and present examples.
We then prove
compactness of bilinear paraproducts with \CMO symbols
by combining one of the properties
of compact bilinear operators thus obtained
 with vanishing Carleson measure estimates and interpolation of bilinear compactness.

\end{abstract}

%
\maketitle

\tableofcontents

\section{Introduction}

The concept of compactness in the context of general multilinear operators was defined in Calder\'on's seminal work on interpolation \cite{Cal}. However, outside interpolation theory, the first manifestation of this concept in harmonic analysis appeared much later in the work~\cite{BenTor1}
by  the first and fourth authors
 who proved that commutators of bilinear Calder\'on-Zygmund operators with \CMO functions are compact from $L^p(\R^d) \times L^q(\R^d) $ into $L^r(\R^d)$ for appropriate exponents $p, q, r$, thus extending the classical result of Uchiyama \cite{Uch} to the bilinear
setting;  see also \cite{BO}
in the context of
bilinear pseudodifferential operators.
 Various generalizations and variations have followed, and the concept of bilinear compactness has taken on a life of its own within this area of research. For an overview (certainly not exhaustive) of recent results on commutators of several classes of bilinear operators in harmonic analysis, see
 the survey paper \cite{BenTor2}.

A fundamental result in the theory of linear Calder\'on-Zygmund operators is
the celebrated  $T(1)$ theorem due to David and Journ\'e \cite{DavJou}, which 
 states that a singular integral operator~$T$ with a  Calder\'on-Zygmund  kernel is bounded if and only if 
it satisfies a certain weak boundedness property (WBP) and $T(1)$ and $T^*(1)$  are functions in 
\BMO (when properly defined).
Here,  $T^*$ denotes the formal transpose of $T$.
In the same paper (see~\cite[p.\,380]{DavJou}),
David and Journ\'e  presented 
another equivalent and extremely elegant statement
that avoids mentioning the WBP, 
 based on controlling the action of $T$  on the omnipresent character functions in harmonic analysis,  $x \mapsto e^{ix\cdot \xi}$ for all $\xi \in \R^d$.
A simplified  proof was then presented by Coifman and Meyer \cite{CM}, which was followed by several 
wavelet-based proofs. 
Finally, Stein~\cite{Stein} provided a 
quantitative statement of the $T(1)$ theorem only in terms of appropriate $L^2$-estimates, which completely avoids the mentioning of the WBP and $\BMO$\,.
Nonetheless,  in his proof, 
 both the WBP and \BMO conditions are still used in some form. 
This  version of a $T(1)$ theorem by  Stein is 
based on controlling 
 the action of $T$ and $T^*$
on normalized bump functions,
 which can be more directly verified in some applications.  
It is important to mention that all these different arguments employ in one way or another the construction of paraproduct operators 
which reduce the matter
to the  particular case of a simpler operator $T$ satisfying $T(1)=T^*(1)=0$. 
The proof  of  boundedness of paraproduct operators by a direct method without using the $T(1)$ theorem is then key.

In the multilinear setting, the first partial version of the $T(1)$ theorem was obtained by Christ and Journ\'e \cite{CJ}, while the full result 
 \cite{GraTor}
is due to Grafakos and the last named author of this article.  
In \cite{GraTor},  
the result was proved using the multilinear version of the control on exponentials and through an iterative process, relying on Stein's $T(1)$ theorem in the linear setting.
In particular, 
the formulation in \cite{GraTor}
was not
in a truly multilinear analogue of the original formulation in ~\cite{DavJou}. 
A version of the bilinear $T(1)$ theorem 
closest  
to that 
in \cite{DavJou} is due to Hart \cite{Har}.

Interestingly, the study of compactness of commutators in the multilinear setting brought back a lot of 
attention to  results involving the notion of compactness even in the linear setting. The literature nowadays has an abundance of harmonic analysis  results related to compactness of commutators in a plethora of different settings
such as  compact weighted estimates,  compact extrapolation, and compact wavelet representations, 
 in both the linear and multilinear cases, and also numerous extensions of the classical Kolmogorov-Riesz
 compactness theorem (a main tool for proving compactness; 
 see, for example, \cite{HHM}). See again \cite{BenTor2} for a survey on these  extensions.
In Section \ref{SEC:2},  we look at some further properties of bilinear compact operators which are perhaps surprising when compared to the linear situation. For example,
 we show that for a bilinear compact operator on general Banach spaces,
  it is not necessarily the case that its transposes are also compact;
  see Proposition \ref{PROP:1}.

Compact Calder\'on-Zygmund operators exist, but most examples are provided by operators arising in the context of layer potential  techniques on smooth bounded domains 
and by those artificially constructed, 
and their compactness can be easily established directly. 
Perhaps,  one notable exception is the class of pseudodifferential operators introduced by Cordes \cite{Cordes} and revisited recently in the weighted setting in \cite{CarSorTor}.  The other important exception is provided by paraproduct operators with appropriate symbols, which we will revisit here in the bilinear setting.

It is natural to expect that  compactness of paraproduct operators would play a crucial role in the proof of a $T(1)$
compactness theorem. This is in fact the case, as it was established in the first  version of such a theorem by Villarroya \cite{V}, which makes some  additional assumptions  on the kernel of a
\CZ  operator. The recent works by Mitkovski and Stockdale~\cite{MitSto} in the linear case (see also Remark \ref{REM:lin} below)
and by Fragkos, Green, and Wick \cite[Theorems~1 and~2]{FraGreWic} in the multilinear case present $T(1)$ compactness results for Calder\'on-Zygmund operators that have a similar flavor to the original $T(1)$ theorem;
see also recent results \cite{BLOT2, CLSY}.
%
Restricting ourselves to the bilinear case, the aforementioned result from~\cite{FraGreWic} is as follows.

\begin{oldtheorem}\label{THM:A}

Let $T: \S(\R^d) \times \S(\R^d) \to \S'(\R^d)$ be a bilinear singular integral operator with a standard Calder\'on-Zygmund kernel, and $1 < p, q \le \infty$ and $\frac 12 < r < \infty$ such that $\frac{1}{p}+\frac{1}{q}=\frac{1}{r}$. Then, $T$ is a compact operator from $L^p(\R^d) \times L^q(\R^d)$ to $L^r(\R^d)$ if and only if
\begin{itemize}
\item[\textup{(i)}]
$ T $ satisfies the  weak compactness property, 
and

\smallskip

\item[\textup{(ii)}]
$ T(1, 1), \,
T^{*1}(1, 1), $
and $ T^{*2}(1, 1)$
are in
$\CMO$.
\end{itemize}

\end{oldtheorem}

In other words, as compared to the bilinear $T(1)$ theorem from \cite{Har}, the weak boundedness property is replaced by an appropriate weak compactness property, while the requirement of~$T$ and its transposes acting on the constant function 1 to belong to \BMO is now replaced by the stronger assumption of belonging to
$\CMO$.  The appearance of $\CMO$ (see Section~\ref{SEC:3} for its definition) is very natural as this space appears in other results related to compactness, starting from the result in \cite{Uch}.

As in the case of the $T(1)$ theorem for boundedness, a main ingredient in  the proof of Theorem A  (and similarly in its linear versions) is to reduce the study of the operator $T$ to that of
$\wt T$ given by
\[
\wt T=T-\Pi_{T(1, 1)}-\Pi_{T^{*1}(1, 1)}^{*1}-\Pi_{T^{*2}(1, 1)}^{*2},
\]

\noi
where $\Pi_b$ denotes an appropriately defined bilinear paraproduct
(satisfying \eqref{para1}), and then realize the operator $\wt T$ as a sum of compact wavelet ones. The reduction from $T$ to $\wt T$ via paraproducts is employed in \cite{Har} as well, the difference being that for the boundedness of~$\wt T$ one can appeal to bilinear square function estimates.
Thus, as already alluded to, the understanding of boundedness or compactness of  bilinear paraproducts is of paramount importance in both the classical multilinear $T(1)$ theorem and its compact $T(1)$ counterpart;
see \cite[Section 5]{FraGreWic}.
See also
\cite[Section 6]{V}
and
\cite[Section 4]{MitSto} in the linear case.

The original goals of this work were more ambitious than what we present here. However, while working on this article, we became aware of the results in \cite{FraGreWic}, which address some of our initial questions about  bilinear compact  $T(1)$ theorems. 
Hence, our modest goal of this short note is to revisit only the compactness of multilinear paraproducts 
with \CMO symbols through 
a different lens than the one in \cite[Section 5]{FraGreWic}, 
namely, by exploring and using  more delicate properties of compact bilinear operators on Banach spaces
which are of interest on their own;
see  Section~\ref{SEC:2}.
Our result  (Proposition \ref{PROP:para}) and its proof in Section~\ref{SEC:3} should be construed
as a compact counterpart
of  \cite[Lemma 5.1]{Har}
on the  boundedness of paraproducts;
the additional ingredients in our argument will be the vanishing of the appropriate Carleson measure as well as the use of interpolation\footnote{While it may be more appropriate
to use the term ``extrapolation'' as in
\cite{COY, HL2, HL1},
we follow \cite{CobFer, CobFerMar}
and use the term ``interpolation''.}
for compact bilinear operators from the work of Cobos, Fern\'andez-Cabrera, and Mart\'inez~\cite{CobFerMar}. For the ease of notation, we will consider only the bilinear case
but interested readers may extend 
the results to a more general multilinear setting.

\section{Some subtle properties of compact bilinear operators}
\label{SEC:2}

Given a metric space $M$,
we use $B_r^M(x)$ to denote the closed ball (in $M$) of radius $r > 0$
centered at $x \in M$.
When it is centered at the origin $x = 0$,
we simply write $B_r^M$ for $ B_r^M(0)$.
When there is no confusion, we drop the superscript $M$
and simply write $B_r(x)$ and $B_r$.

Let $X$, $Y$, and $Z$ be normed vector spaces.
Recall from  \cite{Cal, BenTor1}
that  we say that a bilinear operator $T: X\times Y\to Z$ is a compact bilinear operator
if the image $T(B_1^X\times B_1^Y)$ is precompact in $Z$.
Several equivalent characterizations
of compactness for a bilinear operator $T: X\times Y\to Z$
 are stated in \cite[Proposition 1]{BenTor1}.
 In this section, we
 explore further properties
 of compact bilinear operators
 by comparing them with the corresponding properties of  compact linear operators.
Before proceeding further, let us set some notations.
 We use $\jb{\cdot, \cdot}$ to denote the usual dual pairing; the spaces
  to which the duality pairing applies will be clear from the context. We define the two transposes of $T$ as $T^{*1}: Z^*\times Y\to X^*$ and $T^{*2}: X\times Z^*\to Y^*$ via
\begin{align}
\jb{T(x, y), z^*}=\jb{T^{*1}(z^*, y), x}=\jb{T^{*2}(x, z^*), y}
\label{trans1}
\end{align}

\noi
for all $x\in X$, $y\in Y$ and $z^*\in Z^*$.
Given a  bilinear operator
$T: X \times Y \to Z$,
we define its section operators
 $T_x:Y\to Z$ for fixed $x \in X$
 and
 $T_y:X\to Z$ for fixed $y \in Y$
by setting
\begin{align}
T_x(y) = T(x, y), \ \ y \in Y
\qquad \text{and}\qquad
T_y(x) = T(x, y), \ \ x \in X.
\label{trans2}
\end{align}

\noi
Note that bilinearity of $T$ is equivalent to
linearity of both $T_x$ and $T_y$ for any $x \in X$
and $y \in Y$.
We say that a  bilinear operator
$T: X \times Y \to Z$ is

\smallskip
\begin{itemize}

\item
 separately continuous
if $T_x$ and $T_y$
are continuous linear operators
for any $x \in X$
and $y \in Y$,

\smallskip
\item
 separately compact
if $T_x$ and $T_y$
are compact linear operators
for any $x \in X$
and $y \in Y$.

\end{itemize}

\smallskip

\noi
If $X$ or $Y$ is Banach,
then joint continuity of $T$
is equivalent to separate continuity of $T$;
\cite[Theorem~2.17]{Rudin}.
 The completeness of one of the spaces in the domain of $T$ is crucial for this equivalence.
However, the notion of separate compactness
is strictly weaker than the notion of (joint) compactness
and
 it turns out that  the assumption of completeness of the spaces $X$ and $Y$ is of no importance.
In  \cite[Example 4]{BenTor1}, 
an example of a separately compact bilinear
operator which is not even continuous
(and hence not compact) is provided, where  the spaces are not complete in the relevant topologies.
In Example \ref{EX:3} below, we present
a separately compact bilinear operator which is  continuous but not
 compact,  
 where all the spaces involved are Banach.

\medskip

We first recall the following characterizations
for compact linear operators.

\begin{lemma}\label{LEM:com}
Let $X$ and $Y$ be Banach spaces
and $T:X\to Y$ be a continuous linear operator.

\smallskip

\begin{itemize}
\item[\textup{(i)}]
If $T$ is compact,
then $T$ maps  weakly convergent sequences
to strongly convergent sequences.
Moreover, by assuming in addition that $X$ is reflexive,
if $T$ maps  weakly convergent sequences
to strongly convergent sequences,
then $T$ is compact.

\smallskip

\item[\textup{(ii)}]
The operator
$T$ is compact if
and only if its transpose $T^*$ is compact.

\end{itemize}

\end{lemma}

As for the first claim in Part (i), see
\cite[Theorem VI.11]{RS}.\footnote{For this part,
we do not need to assume that $X$ and $Y$ are Banach spaces.
The result holds for normed vector spaces $X$ and $Y$.}
The second claim in Part~(i) follows from \cite[Definition on p.\,199]{RS},
saying that $T$ is compact if and only if
for any bounded sequence $\{x_n\}_{n \in \N} \subset X$,
the sequence $\{T(x_n)\}_{n \in \N}$ has a convergent subsequence in $Y$,
and that a bounded sequence
$\{x_n\}_{n \in \N} \subset X$ has a weakly convergent subsequence
under the extra assumption that $X$ is reflexive.
As for Part (ii),
see
\cite[Theorem VI.12 (c)]{RS}.

By drawing an analogy to the linear case above,
we investigate the following questions.

\begin{question}\label{Q:1}
Let $X$, $Y$, and $Z$ be Banach spaces
and $T:X\times Y \to Z$ be a continuous bilinear operator.
Do any of the following statements hold true in the bilinear setting?

\begin{itemize}
\item[\textup{(i)}]
If $T$ is compact,
then
for every sequence $\{(x_n, y_n)\}_{n\in \N}\subset X\times Y$
with $\{x_n\}_{n\in \N}$ weakly convergent in $X$ and $\{y_n\}_{n\in \N}$ weakly convergent in $Y$, the sequence $\{T(x_n, y_n)\}_{n\in \N}$ is strongly convergent in $Z$.
By assuming in addition
that $X$ and $Y$ are reflexive,
if for every sequence $\{(x_n, y_n)\}_{n\in \N}\subset X\times Y$
with $\{x_n\}_{n\in \N}$ weakly convergent in $X$ and $\{y_n\}_{n\in \N}$ weakly convergent in $Y$, the sequence $\{T(x_n, y_n)\}_{n\in \N}$ is strongly convergent in $Z$,
then $T$ is compact.

\smallskip

\item[\textup{(ii)}]
The operator  $T$ is compact if and only if $T^{*1}$ is compact if and only if $T^{*2}$ is compact.

\end{itemize}

\end{question}

As we see below,
except for the second statement in Part (i),
the answer is negative in general,
exhibiting a sharp contrast to the linear case
(Lemma \ref{LEM:com}).
In the context of
bilinear Calder\'on-Zygmund operators
from
$L^p(\R^d) \times L^q(\R^d)$ into $L^r(\R^d)$ with $\frac{1}{p}+\frac{1}{q}=\frac{1}{r}$,
however,
the answers to Question \ref{Q:1}\,(i) and (ii)
turn out to be positive
(at least in the reflexive case $1 < p, q, r  < \infty$);
see Proposition  \ref{PROP:4}.
While we restrict our attention only to the bilinear case
in the following,
the discussion (in particular, Propositions \ref{PROP:2}, \ref{PROP:1},
and \ref{PROP:4})
easily extends to the general $m$-linear case.

The next proposition provides
an answer to Question \ref{Q:1}\,(i).

\begin{proposition}
\label{PROP:2}
Let $X$ and $Y$ be  Banach spaces,
 $Z$  be a normed vector space,
 and
$T: X\times Y\to Z$ be a continuous bilinear operator.

\smallskip

\begin{itemize}
\item[\textup{(i)}]
In addition, assume
that $X$ and $Y$ are reflexive.
If for every sequence $\{(x_n, y_n)\}_{n\in \N}\subset X\times Y$
with $\{x_n\}_{n\in \N}$ weakly convergent in $X$ and $\{y_n\}_{n\in \N}$ weakly convergent in $Y$, the sequence $\{T(x_n, y_n)\}_{n\in \N}$ is strongly convergent in $Z$,
then $T$ is compact.

\smallskip

\item[\textup{(ii)}]
The converse of Part (i) is false.

\smallskip

\item[\textup{(iii)}]
If $T$ is compact,
then
 for every sequence $\{(x_n, y_n)\}_{n\in \N}\subset X\times Y$
with $\{x_n\}_{n\in \N}$ weakly convergent in $X$ and $\{y_n\}_{n\in \N}$ weakly convergent in $Y$, the sequence $\{T(x_n, y_n)\}_{n\in \N}$
has a strongly convergent subsequence  in $Z$.

\end{itemize}

\end{proposition}

In Section \ref{SEC:3},
we will use
Proposition \ref{PROP:2}\,(i) in proving
compactness of a bilinear paraproduct;
see Proposition \ref{PROP:para}.

\begin{proof}
(i) Let $\{(x_n, y_n)\}_{n\in \N}$ be a bounded sequence in $X\times Y$.
Our goal is  to construct a subsequence  whose image under $T$ is convergent in $Z$.
Since $X$ is a reflexive Banach space and $\{x_n\}_{n\in \N}$ is bounded in $X$,
it follows from  the Banach-Alaoglu and Eberlein-\v{S}muljan theorems
that there exists a subsequence $\{x_{n_j}\}_{j\in \N}$ that is weakly convergent in $X$.
By the reflexivity of $Y$ and the boundedness of
$\{y_{n_j}\}_{j\in \N}$,
we can  extract a further subsequence $\{y_{{n_{j_k}}}\}_{k\in \N}$ that is weakly convergent in $Y$.
Then, by the hypothesis, the sequence $\{T(x_{{n_{j_k}}}, y_{{n_{j_k}}})\}_{k\in \N}$
is strongly convergent in $Z$.
Hence, from  \cite[Proposition 1 (c7)]{BenTor1},
we conclude that  $T$ is compact.

\medskip

\noi
(ii) See Examples  \ref{EX:1} and \ref{EX:2} below.

\medskip

\noi
(iii)
Fix a  sequence $\{(x_n, y_n)\}_{n\in \N}\subset X\times Y$
such that  $\{x_n\}_{n\in \N}$ is weakly convergent in $X$ and $\{y_n\}_{n\in \N}$
is weakly convergent in $Y$.
Then, $\{(x_n, y_n)\}_{n \in \N}$ is bounded in $X \times Y$.
Hence, it follows from the compactness of $T$ and  \cite[Proposition 1 (c7)]{BenTor1}
that there exists a subsequence $\{T(x_{n_j}, y_{n_j})\}_{j \in \N}$
converging strongly in $Z$.
\end{proof}

\begin{remark}\label{REM:x}\rm
In view of  the bilinearity of $T$,
in Proposition \ref{PROP:2}\,(i),
 it is enough to verify that for all sequences for which at least one of
$\{x_n\}_{n\in \N}$ or $\{y_n\}_{n\in \N}$
converges weakly to $0$,
$\{T(x_n, y_n)\}_{n\in \N}$  converges strongly to $0$ in $Z$,
 to imply that $T$ is compact.
Note that it is {\it not} sufficient
to assume that {\it both}
$\{x_n\}_{n\in \N}$ and $\{y_n\}_{n\in \N}$
converge weakly to $0$
(and showing
that $\{T(x_n, y_n)\}_{n\in \N}$  converges strongly to $0$ in $Z$).

\end{remark}

Part (i) of the next proposition provides a negative answer to Question \ref{Q:1}\,(ii),
showing that, regarding compactness,\footnote{Compare this
with continuity; $T$ is continuous
if and only if $T^{*1}$ is continuous
if and only if $T^{*2}$ is continuous.}
  the bilinear case is quite different from the linear case
(Lemma~\ref{LEM:com}\,(ii)).

\begin{proposition}\label{PROP:1}
\textup{(i)}
There exist Banach spaces $X$, $Y$, and $Z$
and a compact  bilinear operator
 $T:X\times Y \to Z$
such that neither $T^{*1}$ nor $T^{*2}$ is compact.

\smallskip

\noi
\textup{(ii)}
Let $X$, $Y$, and $Z$ be Banach  spaces.
A bilinear operator
 $T:X\times Y \to Z$
 is separately compact
if and only if $(T^{*1})_y$ and $(T^{*2})_x$ are compact for any  $(x, y) \in X \times Y$.
Here, $(T^{*1})_y$ and $(T^{*2})_x$
are the section operators \textup{(}of the transposes\textup{)}
defined in~\eqref{trans2}.

\end{proposition}

\begin{proof}
(i)
See Examples \ref{EX:1} and \ref{EX:2} below.

\smallskip

\noi
(ii)
Suppose that
$(T^{*1})_y$ and $(T^{*2})_x$ are compact for any  $(x, y) \in X \times Y$.
From \eqref{trans1} and \eqref{trans2}, we have
\begin{align*}
\jb{(T^{*1})_y(z^*), x}
= \jb{T^{*1}(z^*, y), x}
= \jb{z^*, T(x, y) }
= \jb{ z^*, T_y(x)}
=
\jb{(T_y)^*(z^*), x}
\end{align*}

\noi
for any $x\in X$, $y \in Y$, and $z^* \in Z^*$.
Hence, together with a similar computation for $(T^{*2})_x$, we have
\begin{align}
(T^{*1})_y = (T_y)^*, \ \  y \in Y
\qquad \text{and}\qquad
(T^{*2})_x = (T_x)^*,
\ \ x \in X.
\label{trans3}
\end{align}

\noi
Then, it follows from Lemma \ref{LEM:com}\,(ii)
with the compactness of
$(T^{*1})_y$ and $(T^{*2})_x$
that $T_y$ and $T_x$ are compact
for any  $(x, y) \in X \times Y$,
which implies separate compactness of $T$ by definition.

Conversely, if $T$ is separately compact,
then
 $T_y$ and $T_x$ are compact
for any  $(x, y) \in X \times Y$.
Hence, from
Lemma \ref{LEM:com}\,(ii) with \eqref{trans3},
we conclude that
 $(T^{*1})_y$ and $(T^{*2})_x$ are compact for any  $(x, y) \in X \times Y$.
\end{proof}

We point out that if $X$ is finite-dimensional,
then $T$ being compact implies $T^{*1}$ is compact.
In this case,  $X^*$ is also finite-dimensional and thus is reflexive.
By noting that
given a sequence $\{(z_n^*, y_n)\}_{n\in \N}\subset B_1^{Z^*}\times B_1^Y$,
 $\{T^{*1}(z_n^*, y_n)\}_{n\in \N}$
 is bounded in $X^*$ and hence
we can extract a
 convergent subsequence,
 which implies compactness of $T^{*1}$.
Similarly,
 if $Y$ is finite-dimensional,
then $T$ being compact implies $T^{*2}$ is compact.
As we see in Example \ref{EX:1},
however, finite dimensionality of the target space $Z$
does not yield compactness of $T^{*1}$ or $T^{*2}$.

\medskip

We now present two examples,
providing proofs
of Proposition \ref{PROP:2}\,(ii)
and Proposition~\ref{PROP:1}\,(i).

\begin{example}\label{EX:1} \rm

Let $X = Y = L^2(\T)$ and $Z = \C$.
Define a bilinear operator $T:X\times Y \to Z$
with $\T = \R/\Z$
by setting
\begin{align*}
T(e_n, e_m) 
= \begin{cases}
1, & \text{if } n + m =0,\\
0, & \text{otherwise},
\end{cases}
 \quad n, m \in \Z
\end{align*}

\noi
and extending the definition bilinearly,
where $e_n(t) = e^{2\pi  int}$, $t \in \T$.
Namely, we have
\[ T(x, y) = \int_\T x(t) y(t) dt.\]

\noi
Then, by Cauchy-Schwarz's inequality
and noting $T(e_n, e_{-n}) = 1$, $n \in \Z$,
we have $\|T\| = 1$, namely $T$ is bounded.
Moreover, $T$ is compact
since
$T(B_1^X \times B_1^Y)
= B_1^Z$ is compact in $Z = \C$.

We first present a proof of  Proposition \ref{PROP:2}\,(ii).
Define a sequence
$\{(x_n, y_n)\}_{n\in \N}\subset X\times Y
= L^2(\T) \times L^2(\T)$
by setting
$x_n = e_{n}$
and $y_n = e_{- n + p(n)}$,
where $p(n)$
denotes the ``parity'' of $n$ given by
\begin{align}
p(n) =
\begin{cases}
1,  & \text{if $n$ is odd},\\
0,  & \text{if $n$ is even}.
\end{cases}
\label{T3}
\end{align}

\noi
By the Riemann-Lebesgue lemma,
we see that both $\{x_n\}_{n \in \N}$
and $\{y_n\}_{n \in \N}$
converge weakly to 0 as $n \to \infty$.
On the other hand, we have
\begin{align*}
T(x_n, y_n) =
\begin{cases}
0, & \text{if $n$ is odd,}\\
1, & \text{if $n$ is even},
\end{cases}
\end{align*}

\noi
which shows that $\{T(x_n, y_n)\}_{n \in \N}$
is not convergent.
This proves
Proposition \ref{PROP:2}\,(ii).

Next, we present a proof of
Proposition \ref{PROP:1}\,(i).
We only show that
$T^{*1}$ is not compact
since non-compactness of
$T^{*2}$ follows from
a similar argument.
It follows from \cite[Proposition~1\,(c7)]{BenTor1}
that if $T^{*1}$ were compact,
then given any bounded sequence
 $\{(z_n^*, y_n)\}_{n\in \N} \subset Z^* \times Y$,
 there would exist a subsequence $\{T^{*1}(z_{n_j}^*, y_{n_j})\}_{j \in \N}$
 that is strongly convergent in $X^*$.
We will show that this property fails.

Define a bounded sequence  $\{(z_n^*, y_n)\}_{n\in \N}\subset B_1^{Z^*}\times B_1^Y$
by setting $z_n^* = 1$  and
$y_n = e_n$, $n \in \N$.
Pick an arbitrary subsequence
$\{(z_{n_j}^*, y_{n_j})\}_{j\in \N}$.
Then, by
the definition of a dual norm and \eqref{trans1}, we have
\begin{align}
\begin{split}
& \|T^{*1}(z_{n_j}^*, y_{n_j})-T^{*1}(z_{n_k}^*, y_{n_k})\|_{X^*}\\
&\quad
=  \sup_{x\in B_1^X}
|\jb{T^{*1}(1, y_{n_j}), x}-\jb{T^{*1}(1, y_{n_k}), x}| \\
&\quad =\sup_{x\in B_1^X}
|\jb{T(x, y_{n_j}), 1}-\jb{T(x, y_{n_k}), 1}| \\
& \quad
\ge 1
\end{split}
\label{T1}
\end{align}

\noi
for any $j >  k \ge 1$,
where the last step follows from choosing
$x =  e_{-n_j}$.
This shows that the subsequence
$\{T^{*1}(z_{n_j}^*, y_{n_j})\}_{j \in \N}$
is not convergent in $X^*$.
Since the choice of the subsequence was
arbitrary,
we conclude that there exists no subsequence
of $\{T^{*1}(z_{n}^*, y_{n})\}_{n \in \N}$
that is strongly convergent in $X^*$
and therefore, $T^{*1}$ is not compact.
This proves
Proposition \ref{PROP:1}\,(i).
\end{example}

We provide another example,
where $Z$ is now infinite-dimensional.

\begin{example} \label{EX:2}\rm
Let $X = Y = L^4(\T)$ and $Z = L^2(\T)$.
Given $s > 0$, define a bilinear operator $T: X \times Y \to Z$ by setting
\begin{align*}
T(x, y)(t) = \jb{\dt}^{-s}(xy)(t).
\end{align*}

\noi
Here,  $\jb{\dt}^{-s} = (1- \dt^2)^{-\frac s2}$ denotes the Bessel potential of order $s > 0$
defined by
\begin{align*}
\jb{\dt}^{-s} f = \sum_{n \in \Z}
\frac1{(1 + 4\pi^2n^2)^\frac s2}
\ft f(n) e_n,
\end{align*}

\noi
where
$e_n(t) = e^{2\pi  int}$ as above and
$\ft f(n)$ denotes the Fourier coefficient of $f$.
Then, by Cauchy-Schwarz's inequality,
we see that $T(x, y) \in H^s(\T)$ for any $x \in X$ and $y \in Y$.
Here, $H^s(\T)$ denotes the standard $L^2$-based Sobolev space.
By the Rellich lemma (see Remark \ref{REM:Rellich} below), the embedding $H^s(\T)\hookrightarrow L^2(\T)$
is compact and hence $T$ is compact.

We first present a proof of  Proposition \ref{PROP:2}\,(ii).
Let
$\{(x_n, y_n)\}_{n\in \N}\subset X\times Y
= L^4(\T) \times L^4(\T)$
by setting
$x_n = e_{n}$
and $y_n = e_{- n + p(n)}$,
where $p(n)$
is as in \eqref{T3}.
Then, we have
\begin{align*}
T(x_n, y_n) =
\begin{cases}
\frac1{(1 + 4\pi^2)^\frac s2}e_1, & \text{if $n$ is odd,}\\
1, & \text{if $n$ is even}.
\end{cases}
\end{align*}

\noi
Namely,  $\{T(x_n, y_n)\}_{n \in \N}$
is not convergent, giving another example
for  Proposition \ref{PROP:2}\,(ii).

Next, we present a proof of
Proposition \ref{PROP:1}\,(i).
Choose a bounded sequence  $\{(z_n^*, y_n)\}_{n\in \N}\subset B_1^{Z^*}\times B_1^Y$
by setting $z_n^* = 1$  and
$y_n = e_n$, $n \in \N$.
Then, the computation in~\eqref{T1} holds
by choosing
$x =  e_{-n_j}$,
(where the duality pairing is re-interpreted accordingly),
which shows that
$T^{*1}$ is not compact.
A similar argument shows that $T^{*2}$
is not compact either.

\end{example}

\begin{remark}\label{REM:Rellich}
\rm

The Rellich lemma 
on the circle (namely, 
the compactness of
 the embedding $H^s(\T)\hookrightarrow L^2(\T)$
for $s >  0$) is well known
(see, for example,  \cite[(3.12) and Proposition 3.4]{Taylor})
and widely used (see, for example,  \cite[Remark~1.2]{Molinet}
and \cite[Section 4]{OS}).
In the following, 
we present an elementary proof  for readers' convenience.
In view of Lemma~\ref{LEM:com}\,(i), 
it suffices to show that any weakly convergent sequence in $H^s(\T)$, $s > 0$, 
is strongly convergent in $L^2(\T)$.

Given $s > 0$, let $\{x_n\}_{n \in \N} \subset H^s(\T)$
be  weakly convergent  in $H^s(\T)$.
Without loss of generality, we assume that $x_n$ 
converges weakly to $0$ in $H^s(\T)$
and that $\sup_{n \in \N}\|x_n\|_{H^s} \le 1$.
Given small $\eps > 0$, 
choose $K = K(\eps)  \in \N$ such that 
\begin{align}
(1 + 4\pi^2K^2)^{-\frac s2}  <\eps.
\label{X1}
\end{align}

\noi
On the other hand, the weak convergence of 
$\{x_n\}_{n \in \N}$ to $0$ in $H^s(\T)$ (and thus in $L^2(\T)$) implies that 
there exists $N = N(\eps) \in \N$ such that 
\begin{align}
|\jb{x_n, e_k}_{L^2}| < (2K+1)^{-1} \eps
\label{X2}
\end{align}

\noi
for any $n \ge N$ and any $k \in \Z$ with $|k| \le K$.
Then, by 
Cauchy-Schwarz's inequality, \eqref{X2}, and~\eqref{X1}
with $\sup_{n \in \N}\|x_n\|_{H^s} \le 1$, 
 we have 
\begin{align*}
\| x_n\|_{L^2} 
& = \sup_{\| \phi\|_{L^2} = 1} |\jb{x_n, \phi}_{L^2}|\\
& \le 
\sum_{|k|\le K} |\jb{x_n, e_k}_{L^2}|
+ 
\sup_{\| \phi\|_{L^2} = 1}
\sum_{|k|> K}(1 + 4\pi^2k^2)^{-\frac s2}
\big((1 + 4\pi^2k^2)^\frac s2
\ft x_n(k)\big) \cj{\ft \phi(k)}\\
& < 2\eps
\end{align*}

\noi
for any $n \ge N$, 
which shows that $\{x_n\}_{n \in \N}$
converges strongly to $0$ in $L^2(\T)$.

\end{remark}

 The next  example provides
a continuous bilinear operator
that is separately compact but is not jointly compact, even in the Hilbert space setting

\begin{example}\label{EX:3} \rm
Let $X = Y = Z = \l^2(\N)$.
Given $n \in \N$,
let $\dl^n$ be the $n$th basis element in $\l^2(\N)$
whose only non-zero entry appears in the $n$th place
and is given by  $1$.
Define a bilinear operator $T:X\times Y \to Z$
by setting
\begin{align*}
T(x, y) =
\sum_{n = 1}^ \infty x_n y_n\dl^n =
(x_1 y_1, x_2y_2, \dots)
\end{align*}

\noi
for $x = \{x_n\}_{n \in \N}$
and $y = \{y_n\}_{n \in \N}$.
By H\"older's inequality and the embedding $\l^2(\N) \subset \l^\infty(\N)$, we
have
\begin{align*}
\|T(x, y)\|_{\l^2} \le \| x\|_{\l^\infty} \| y\|_{\l^2}
\le \| x\|_{\l^2} \| y\|_{\l^2}.
\end{align*}

\noi
Moreover, we have $T(\dl^n, \dl^n) = \dl^n$, $n \in \N$,
and thus
 $T$ is bounded with $\|T\|= 1$.

We first show that $T$ is separately compact.
Given $N \in \N$,
define the projection $\P_N$
by setting
$\P_N x = \sum_{n = 1}^N x_n \dl^n$.
Then,
it follows from the dominated convergence theorem that
\begin{align}
\begin{split}
\| T(x, y)  - \P_N T(x, y) \|_{\l^2}
& = \bigg\|  \sum_{n = N+1}^\infty  x_n y_n \dl^n \bigg\|_{\l^2}
\le \| x\|_{\l^2} \bigg( \sum_{n = N+1}^\infty |y_n|^2 \bigg)^\frac 12 \\
& \too 0,
\end{split}
\label{T4}
\end{align}

\noi
as $N \to \infty$, uniformly
in $x \in B_1^{\l^2}$.
Hence, from \eqref{trans2} and \eqref{T4},
we see that $T_y$ is the limit (in the operator norm topology) of
finite rank operators $(\P_NT)_y$
for each $y \in Y$,
which implies that $T_y$ is compact
for any $y \in Y$.
By symmetry, we deduce that $T_x$ is also compact
for any $x \in X$.
This shows that  $T$ is separately
 compact.

Next, we show that $T$ is not compact.
Noting that $T(\dl^n, \dl^n) = \dl^n$, $n \in \N$,
and
that $\| \dl^n - \dl^m\|_{\l^2} = \sqrt2$
for any $n \ne m$,
we see that
the sequence
$\{(\dl^n, \dl^n) \}_{n \in \N}$
is bounded in $X\times Y = \l^2(\N)\times \l^2(\N)$
but that
$\{T(\dl^n, \dl^n) \}_{n \in \N}$
does not have any convergent subsequence
in $Z = \l^2(\N)$.
In view of \cite[Proposition~1\,(c7)]{BenTor1},
this shows  non-compactness of $T$.

By working on the Fourier side,
the  argument above shows that
for $X = Y = Z = L^2(\T)$,
the operator
 $S$ defined  by
\[ S(x, y)(t) = x*y (t) = \int_\T x(t - s) y(s) ds \]

\noi
 is continuous and
separately compact but is not (jointly) compact.

\end{example}

We conclude this section by
discussing the case of
 bilinear Calder\'on-Zygmund operators.
In  the reflexive case $1 < p, q, r  < \infty$,
 the following proposition
(together with Proposition~\ref{PROP:2}\,(i))
provides positive answers to
 Question \ref{Q:1}\,(i) and (ii).

\begin{proposition}\label{PROP:4}

Let $T: \S(\R^d) \times \S(\R^d) \to \S'(\R^d)$ be a bilinear singular integral operator with a standard Calder\'on-Zygmund kernel.  
Then, the following statements hold
for any
$1 < p, q, r  < \infty$ with
$\frac{1}{p}+\frac{1}{q}=\frac{1}{r}$.

\begin{itemize}
\item[\textup{(i)}]
The operator $T: L^p(\R^d) \times L^q(\R^d)\to L^r(\R^d)$
is compact if and only if
$T^{*1}: L^{r'}(\R^d) \times L^q(\R^d)\to L^{p'}(\R^d)$
is compact if and only if
$T^{*2}: L^{p}(\R^d) \times L^{r'}(\R^d)\to L^{q'}(\R^d)$.

\smallskip

\item[\textup{(ii)}]
Suppose that $T$ is
compact from $L^p(\R^d) \times L^q(\R^d)$ to $L^r(\R^d)$.
Then,
 for every sequence $\{(f_n, g_n)\}_{n\in \N}\subset L^p(\R^d) \times L^q(\R^d)$
with $\{f_n\}_{n\in \N}$ weakly convergent  in $ L^p(\R^d) $ and $\{g_n\}_{n \in \N}$
weakly convergent in $ L^q(\R^d) $, the sequence $\{T(x_n, y_n)\}_{n\in \N}$ is strongly convergent in $ L^r(\R^d) $.

\end{itemize}

\end{proposition}

\begin{proof}
(i)
We only prove that compactness of $T$
implies compactness of $T^{*1}$ and $T^{*2}$.
We first note that the hypotheses
on
  the kernels  and
the weak compactness property
 in  Theorem~\ref{THM:A} are symmetric 
for $T$, $T^{*1}$, and $T^{*2}$.
Moreover,
by noting that\footnote{Here,
we used the reflexivity of   $L^p(\R^d)$, $1 < p < \infty$.}
$(T^{*1})^{*1} = T $ and $(T^{*1})^{*2} = (T_{\text{flip}})^{*1}$,
where $T_{\text{flip}}(f,g) = T(g,f)$,  and
that
if $T$
 is compact from
 $L^p(\R^d) \times L^q(\R^d)$ to $L^r(\R^d)$,
 then  $T_\text{flip}$ is compact from
 $L^q(\R^d) \times L^p(\R^d)$ to $L^r(\R^d)$,
 it follows from Theorem~\ref{THM:A}
 that $T^{*1}(1, 1)$,
$(T^{*1})^{*1}(1, 1)  = T(1, 1) $,  and $(T^{*1})^{*2} (1, 1)= (T_{\text{flip}})^{*1}(1, 1)$
are all in $\CMO$\,.
Hence, by applying Theorem \ref{THM:A}
in the reversed direction, we conclude that $T^{*1}$
is compact from
 $L^{r'}(\R^d) \times L^q(\R^d)$ into $L^{p'}(\R^d)$.
A similar argument shows that
$T^{*2}$
is compact from
 $L^p(\R^d) \times L^{r'}(\R^d)$ into $L^{q'}(\R^d)$.

 \smallskip

\noi
(ii)
Fix a sequence $\{(f_n, g_n)\}_{n\in \N}\subset L^p(\R^d) \times L^q(\R^d)$
such that  $\{f_n\}_{n\in \N}$ converges weakly  to some $f$ in $ L^p(\R^d) $ and $\{g_n\}_{n \in \N}$
converges weakly  to some $g$ in $ L^q(\R^d) $.
Our goal is to show that  the sequence $\{T(x_n, y_n)\}_{n\in \N}$ is strongly convergent in $ L^r(\R^d) $.

In view of the bilinearity of $T$, we have
\begin{align}
T(f_n, g_n)-T(f, g)=T(f, g_n-g)+T(f_n-f, g)+T(f_n-f, g_n-g).
\label{A1}
\end{align}

\noi
Since $T$ is separately compact,
the first two terms on the right-hand side of \eqref{A1}
converge to $0$ in $L^r(\R^d)$ as $n \to \infty$.
Therefore,
it suffices to prove that if $f_n$ converges weakly to $0$ in $L^p(\R^d)$ and $g_n$ converges weakly to $0$ in $L^q(\R^d)$, then $T(f_n ,g_n)$ converges to 0
in $L^r(\R^d)$.

Fix a subsequence $\{T(f_{n_j}, g_{n_j})\}_{j \in \N}$.
We show that it has a further subsequence that converges to $0$ in $L^r(\R^d)$.
For simplicity of notations, set $X = L^p(\R^d)$,
$Y = L^q(\R^d)$, and $Z = L^r(\R^d)$.
Without loss of generality,
we assume that $f_n \in B_1^{X}$ and $g_n \in B_1^{Y}$
for any $n \in \N$.
Note that the closed unit ball $B_1^{Z^*}$ is equicontinuous as a collection
of continuous linear functionals on $Z$. Indeed, for any $h, h'\in Z$ and $h^*\in Z^*$ with $\|h^*\|_{Z^*}\leq 1$, we have
\[
|\jb{h^*, h}-\jb{h^*,  h'}|\leq \|h^*\|_{Z^*}\|h- h'\|_{Z}\leq \|h- h'\|_{Z}.
\]

\noi
Hence, the restriction of $B_1^{Z^*}$ to
a compact set $E:=\overline{T(B_1^X\times B_1^Y)}$,
denoted by  $(B_1^{Z^*})|_E$, is a pointwise-bounded, equicontinuous collection of functions on
a compact set $E$.
By the Arzel\`a-Ascoli theorem, we obtain that
$(B_1^{Z^*})|_E$ is a precompact subset of the space of continuous linear functionals on $E$.

Fix small $\eps > 0$.
Then, for each  $j \in \N$, there exists $h_j \in (B_1^{Z^*})|_E$
such that
\begin{align}
\begin{split}
\|T(f_{n_j} ,g_{n_j})\|_{L^r}
& = \sup_{h\in B_1^{Z^*}} |\jb{ T(f_{n_j} ,g_{n_j}), h}|
 \le |\jb{ T(f_{n_j} ,g_{n_j}), h_j}| + \eps\\
& \le \| T^{*1}(h_j ,g_{n_j})\|_{X^*} + \eps.
\end{split}
\label{A2}
\end{align}

\noi
By the precompactness of $(B_1^{Z^*})|_E$,
we can extract a  subsequence $\{h_{j_k}\}_{k\in \N}\subset(B_1^{Z^*})|_E$
converging to $h_\infty$.
Thus, there exists  $N_1 = N_1(\eps) \in \N$ such that
\begin{align}
\| T^{*1}(h_{j_k} ,g_{n_{j_k}})
- T^{*1}(h_\infty ,g_{n_{j_k}})
\|_{X^*}
\les \| T^{*1}\|\, \|h_{j_k} - h_\infty\|_{Z^*}
< \eps
\label{A3}
\end{align}

\noi
for any $k \ge N_1$.
Lastly, from the compactness of $T$ and Part (i)
of this proposition,
we see that $T^{*1}$ is compact
and thus is separately compact.
Since $\{g_{n_{j_k}}\}_{k \in \N}$
converges weakly to 0 in $Y$,
it follows from Lemma \ref{LEM:com}\,(i) that
 $T^{*1}(h_\infty ,g_{n_{j_k}})$
converges strongly to $0$ in $X^*$
as $k \to \infty$.
In particular, there exists $N_2 = N_2(\eps) \in \N$
such that
\begin{align}
\|T^{*1}(h_\infty ,g_{n_{j_k}})
\|_{X^*} < \eps
\label{A4}
\end{align}

\noi
for any $k \ge N_2$.

Therefore, putting \eqref{A2}, \eqref{A3}, and \eqref{A4}
together,
we conclude that
\begin{align*}
\|T(f_{n_{j_k}} ,g_{n_{j_k}})\|_{L^r} < 3\eps
\end{align*}

\noi
for any $k \ge \max(N_1, N_2)$.
Since the choice of $\eps$ was arbitrary,
we then conclude that
the  subsubsequence $\{T(f_{n_{j_k}}, g_{n_{j_k}})\}_{k \in \N}$
converges strongly to $0$ in $ L^r(\R^d)$.
This shows that
the original sequence $T(f_n ,g_n)$ converges to 0
in $L^r(\R^d)$
as $n \to \infty$.
\end{proof}

%

\section{Bilinear paraproducts with \CMO symbols}
\label{SEC:3}

We first recall the definition of $\BMO\, (\R^d)$, the space of functions of bounded mean oscillation. Given a locally integrable function $f$ on $\R^d$, its $\BMO$\,-seminorm is given by
\[ \|f\|_{\BMO} = \sup_{Q} \frac{1}{|Q|} \int_Q|f(x) - f_Q|dx,\]
	
\noi
where the supremum is taken over all cubes $Q \subset \R^d$
and $f_Q$ stands for the mean of $f$ over $Q$, namely
\[f_Q  = \frac{1}{|Q|} \int_Q f(x) dx.\]

\noi
We say that $f$ is of bounded mean oscillation if $\|f\|_{\BMO}< \infty$, and denote
\[ \BMO\,(\R^d)  =
\big\{ f \in L^1_{\text{loc}}(\R^d):\,
\| f \|_{\BMO} < \infty \big\}.\]


 \noi
 As usual, 
 we view this space as a space of equivalent classes
 of functions modulo additive constants.
The closure of $C_c^\infty (\R^d)$ in the $\BMO$ topology is called the space of functions of continuous mean oscillation, and it is denoted by $\CMO \, (\R^d)$.
In the following,  we suppress the underlying space $\R^d$  from our notation.

Let $\varphi, \psi\in C_c^\infty$ be radial functions such that
$\supp (\varphi)\subset B_1$,
$\widehat\psi (0)=0$, and
\begin{align}
\int_0^\infty |\ft \psi(te_1)|^2\frac{dt}{t}=1,
\label{P0}
\end{align}

\noi	
 where $e_1=(1, 0,..., 0)\in\R^d$. For $t\in\R_{+}$,
we also define  the linear convolution operators $P_t$ and $Q_t$
by $P_t f=\varphi_t *f $ and $Q_t f=\psi_t*f$, where $h_t=t^{-d}h(t^{-1}\,\cdot\,)$
for a function $h$ on $\R^d$.
Then,  the Calder\'on reproducing formula \cite{Cal}
states the following
\begin{align}
\int_0^\infty Q_t^2 f \frac{dt}{t}=f
\label{P0a}
\end{align}

\noi
 in $L^2$,
where  $Q_t^2 f=Q_t(Q_t f)=\psi_t*\psi_t*f$;
see also~\cite{Wil}.

Given $b\in\BMO$,
we now define
a \emph{bilinear paraproduct} $\Pi_b$ by\footnote{Hereafter,  as it is customary,
we avoid a detailed explanation on the sense in which the integrals based on Calder\'on's formula converge to the represented objects. The interested reader  can consult \cite{Wil, Har}  for precise explanations and \cite{BenTor3} for further references.}
\begin{align}
\Pi_b (f, g) =\int_0^\infty Q_t\big((Q_t b)(P_t f)(P_t g)\big) \frac{dt}{t}.
\label{P1}
\end{align}

\noi
We have the following compactness result
on the bilinear paraproduct $\Pi_b$;
 see also \cite[Proposition~5.2]{FraGreWic}.

\begin{proposition}
\label{PROP:para}
Let $1<p, q<\infty$ and $\frac{1}{2}<r<\infty$ be such that $\frac{1}{p}+\frac{1}{q}=\frac{1}{r}.$ If $b\in\CMO$, then
$\Pi_b$ defined in \eqref{P1}  is a compact bilinear Calder\'on-Zygmund operator
from $L^p\times L^q$ into  $L^r$,
satisfying
\begin{align}
\Pi_b (1, 1)=b
\qquad \text{and}\qquad
 \Pi_b^{*j}(1, 1)=0,\quad  j=1, 2.
\label{para1}
\end{align}

\end{proposition}

\begin{proof}
Fix $b \in \CMO$.
Since $b \in \BMO$, it follows from
 \cite[Lemma 5.1]{Har}
 that $\Pi_b$ is a bilinear Calder\'on-Zygmund operator,
 satisfying \eqref{para1},
that is bounded
 from $L^p\times L^q$ into $L^r$ for any  $1<p, q<\infty$ and $\frac 12<r<\infty$
 such that  $\frac 1p+\frac 1q=\frac 1r$.
In the following,
we show that under the stronger assumption $b\in\CMO$,
the bilinear paraproduct   $\Pi_b$ is indeed a compact bilinear operator from $L^p\times L^q$ into $L^r$.

Fix  $2< p, q<\infty$  such that $\frac 1p+\frac 1q=\frac 12$.
We first  show that $\Pi_b$ is compact from $L^p\times L^q$ into $L^2$.
Let $\{(f_n, g_n)\}_{n\in\mathbb N}\subset L^p\times L^q$ such that $f_n$ converges weakly  in $L^p$
and $g_n$ converges weakly  in $L^q$.
Moreover, we assume that either $f_n$ converges weakly to $0$
or $g_n$ converges weakly to $0$
as $n \to \infty$.
Then, our goal is to show
 that $\|\Pi_b (f_n ,g_n)\|_{L^2}$ converges to $0$
 as $n \to \infty$.

We first
note that, since $b\in \CMO$, the non-negative measure $\mu$ defined by
\begin{align}
d\mu(x, t)=|Q_t b(x)|^2dx\frac{dt}{t}
\label{P2}
\end{align}

\noi
is a vanishing Carleson measure on $\R^{d+1} = \R^d \times \R_+$; see \cite[Definition 1.3 and Remark 3.2]{DinMei}.

 Let $h\in L^2$ with $\|h\|_{L^2}\leq 1$.
 By using H\"older's inequality (in $t$),  the  square function estimate:\footnote{In fact, in the current $L^2$ setting,
by using  \eqref{P0a},
the first inequality in \eqref{P3}
is indeed an equality.
One may also prove this fact via
Plancherel's identity and
 the normalizing condition \eqref{P0}
 with the radiality of $\psi$; see \cite[p.\,27]{Stein}.
For the general $L^p$ setting, $1 < p < \infty$,
see  \cite[Subsection I.8.3]{Stein}.}  
\begin{align}
\bigg\|\bigg(\int_0^\infty|Q_t h|^2\frac{dt}{t}\bigg)^{\frac 12}\bigg\|_{L^2}
\les \|h\|_{L^{2}}
\le 1,
\label{P3}
\end{align}
and \eqref{P2},
we obtain
\begin{align}
\begin{split}
 | & \jb{\Pi_b(f_n ,g_n), h}|\\
& \le
\int_0^\infty
\int_{\R^d} \big|\big(Q_t b(x) P_t f_n(x) P_t g_n(x) \big) Q_t h(x)\big| dx \frac{dt}{t}\\
&
\leq \bigg(\int_0^\infty\int_{\R^d}|P_t f_n(x)|^2|P_t g_n(x)|^2|Q_t b(x)|^2 d x \frac{dt }{t}\bigg)^{\frac{1}{2}}
\bigg\|\bigg(\int_0^\infty|Q_t h|^2 \frac{dt}{t}\bigg)^{\frac{1}{2}}\bigg\|_{L^2} \\
&
\lesssim \bigg(\int_0^{\infty} \int_{\R^d}|P_t f_n(x)|^p|Q_t b(x)|^2 d x \frac{d t}{t}\bigg)^{\frac{1}{p}}
\bigg(\int_0^{\infty} \int_{\R^d}|P_t g_n (x)|^q|Q_t b(x)|^2 d x \frac{d t}{t}\bigg)^{\frac{1}{q}}\\
&
=  \|P_t f_n(x)\|_{L^p(\mathbb R_+^{d+1}, d\mu)}\|P_t g_n(x)\|_{L^q(\mathbb R_+^{d+1}, d\mu)},
\end{split}
\label{P4}
\end{align}

\noi
uniformly in $h \in L^2$ with $\|h\|_{L^2} \le 1$.
Since $d\mu$ is a vanishing Carleson measure,
it follows from
\cite[Theorem 2.1]{DinMei} that
the convolution operator $P_t$ is compact from $L^p(\R^d)$
to $L^p(\R^{d+1}_+; d\mu)$
for $1 < p < \infty$.
In view of the weak convergence of $f$ or $g$ to $0$, we
then have
\begin{align}
\|P_t f_n(x)\|_{L^p(\mathbb R_+^{d+1}, d\mu)}\too 0
\qquad \text{or}\qquad \|P_t g_n(x)\|_{L^q(\mathbb R_+^{d+1}, d\mu)}\too 0,
\label{P5}
\end{align}

\noi
as $n \to \infty$.
From \eqref{P4} and \eqref{P5},
we see  that
 $\Pi_b (f_n ,g_n)$ converges  strongly to $0$ in $L^2$.
Hence, from  Proposition \ref{PROP:2}\,(i)
and Remark \ref{REM:x},
 we conclude that the bilinear paraproduct $\Pi_b$ is compact from $L^p\times L^q$ to $L^2$ with $1<p, q<\infty$ satisfying $\frac 1p+\frac 1q=\frac 12$.

Finally,  recalling that  $\Pi_b$ is also bounded from $L^p\times L^q$ to $L^r$ for all $1<p, q<\infty$ and
$\frac 12<r<\infty$ with $\frac 1p+\frac 1q=\frac 1r$,
we conclude  from interpolation of bilinear compactness \cite[Theorem~5.2]{CobFerMar}
(see also the proof  of \cite[Theorem~6.1]{CobFerMar})  that $\Pi_b$ is in fact compact from $L^p\times L^q$ to $L^r$ for all $1<p, q<\infty$ and
$\frac 12<r<\infty$ with $\frac 1p+\frac 1q=\frac 1r$.
\end{proof}

\begin{remark}\rm

In the proof of Proposition \ref{PROP:para},
we needed to assume $p, q < \infty$
in applying
\cite[Theorem~2.1]{DinMei}
on the compactness of
 $P_t$  from $L^p(\R^d)$
to $L^p(\R^{d+1}_+; d\mu)$
and
\cite[Theorem~5.2]{CobFerMar} on
interpolation of bilinear compactness.
Compare this with the situation in \cite{FraGreWic},
where the upper endpoint ($p = \infty$ or $q = \infty$)
is allowed; see \cite[Remark 3.5]{FraGreWic}.

\end{remark}

%
%

\begin{remark} \label{REM:lin} \rm

 In the linear case, the compact $T(1)$ theorem
 in \cite[Theorem 1.1]{MitSto}
 provides an $L^2$-characterization of compact  linear Calder\'on-Zygmund operators.
 By noting that
 a Calder\'on-Zygmund operator is $L^p$-bounded for all $1<p<\infty$,
we see  from  Krasnosel'ski\u{i}'s interpolation result \cite{Kra}
that  the compact $T(1)$ theorem in \cite{MitSto}
is in fact a characterization of $L^p$-compactness for all $1<p<\infty$; see also
\cite[Remark 2.22]{V}.
See
\cite{CobFer} for a discussion of interpolation results for compact linear operators between more general Banach spaces.

\end{remark}

\begin{ackno}\rm

\'A.B.  acknowledges the support from
 an  AMS-Simons Research Enhancement Grant for PUI Faculty.
G.L. and T.O.~were supported by the European Research Council (grant no.~864138 ``SingStochDispDyn").
The first three authors would like to thank
the West University of Timi\c{s}oara
for its
hospitality, where part of this paper was prepared.
The authors would like to thank the anonymous referees for the helpful comments.

\end{ackno}


\begin{thebibliography}{99}


\bibitem{BO}
\'A.~B\'enyi, T.~Oh, \emph{Smoothing of commutators for a H\"ormander class of bilinear pseudodifferential operators,} J. Fourier Anal. Appl. 20 (2014), no. 2, 282--300.




\bibitem{BLOT2}
\'A.~B\'enyi, G.~Li, T.~Oh, R.H.~Torres, 
{\it Compact $T(1)$ theorem \'a la Stein}, 
arXiv:2405.08416 [math.FA].


\bibitem{BenTor1}
\'A.~B\'enyi, R.H.~Torres, \emph{Compact bilinear operators and commutators},
Proc. Amer. Math. Soc. 141 (2013), no. 10, 3609--3621.

\bibitem{BenTor3}
\'A.~B\'enyi, R.H.~Torres, \emph{The discrete Calder\'on reproducing formula of Frazier and Jawerth}, 
Functional analysis, harmonic analysis, and image processing: a collection of papers in honor of Bj\"orn Jawerth, 
79--107,
Contemp. Math., 693, Amer. Math. Soc., Providence, RI, 2017. 


\bibitem{BenTor2}
\'A.~B\'enyi, R.H.~Torres, \emph{An update on the compactness of bilinear commutators},
preprint.




\bibitem{Cal}
A.P.~Calder\'on, \emph{Intermediate spaces and interpolation, the complex method}, Studia Math. 24 (1964), 113--190.


\bibitem{CLSY}
M.~Cao, H.~Liu, Z.~Si, K.~Yabuta
{\it A characterization of compactness via bilinear $T1$ theorem}, 
arXiv:2404.14013 [math.CA].


\bibitem{COY}
M.~Cao, A.~Olivo, K.~Yabuta, \emph{Extrapolation for multilinear compact operators and applications},
Trans. Amer. Math. Soc. 375 (2022), no.7, 5011--5070.




\bibitem{CarSorTor}
M.J.~Carro, J. Soria, R.H~Torres, \emph{Extrapolation of compactness for certain pseudodifferential operators,} Rev. Un. Mat. Argentina 66 (2023), no. 1,177--186.



\bibitem{CJ}
 M.~Christ, D.~Journ\'e, \emph{Polynomial growth estimates for multilinear singular integral operators}, Acta Math. 159 (1987), no. 1-2, 51--80.

\bibitem{CobFer}
F.~Cobos, D.L.~Fernandez, \emph{On interpolation of compact operators,}
Ark. Mat. 27 (1989), no. 2, 211--217.




\bibitem{CobFerMar}
F.~Cobos, L.M.~Fern\'andez-Cabrera, A.~Mart\'inez,
\emph{Interpolation of compact bilinear operators among quasi-Banach spaces and applications},
Math. Nachr. 291 (2018), no. 14-15, 2168--2187.


%
\bibitem{CM}
R.R.~Coifman, Y.F.~Meyer, 
{\it A simple proof of a theorem by G. David and J.-L. Journ\'e on singular integral operators}, Probability theory and harmonic analysis (Cleveland, Ohio, 1983), 61--65, Monogr. Textbooks Pure Appl. Math., 98, Dekker, New York, 1986. 



\bibitem{Cordes}
H.O.~Cordes, \emph{On compactness of commutators of multiplications and convolutions, and boundedness of pseudodifferential operators,} 
 J. Functional Analysis 18 (1975), 115--131. 

\bibitem{DavJou}
G.~David, J.-L.~Journ\'e,
 \emph{A boundedness criterion for generalized Calder\'on-Zygmund operators,}
 Ann. of Math.  120 (1984), no. 2, 371--397.



\bibitem{DinMei}
 Y.~Ding, T.~Mei,
 \emph{Vanishing Carleson measures associated with families of multilinear operators,}
 J. Geom. Anal. 26 (2016), no. 2, 1539--1559.


\bibitem{GraTor}
L.~Grafakos, R.H.~Torres,
\emph{Multilinear Calder\'on-Zygmund theory,}
Adv. Math. 165 (2002), no. 1, 124--164.

\bibitem{FraGreWic}
A.~Fragkos, A.W.~Green, B.D.~Wick, \emph{Multilinear wavelet compact $T(1)$ theorem},
arXiv:2312.09185 [math.CA].


\bibitem{HHM}
H.~Hanche-Olsen, H.~Holden, 
{\it The Kolmogorov-Riesz compactness theorem},
Expo. Math. 28 (2010), no. 4, 385--394. 
Addendum to ``The Kolmogorov-Riesz compactness theorem''
[Expo. Math. 28 (2010) 385--394]. Expo. Math. 34 (2016), no. 2, 243--245.




\bibitem{Har}
J.~Hart, \emph{A new proof of the bilinear T(1) theorem,} Proc. Amer. Math. Soc. 142 (2014), no. 9, 3169--3181.






\bibitem{HL2}
T.~Hyt\"onen, S.~Lappas, \emph{Extrapolation of compactness on weighted spaces: bilinear operators},
 Indag. Math. (N.S.) 33 (2022), no. 2, 397--420.


\bibitem{HL1}
T.~Hyt\"onen, S.~Lappas, \emph{Extrapolation of compactness on weighted spaces},
Rev. Mat. Iberoam. 39 (2023), no. 1, 91--122.

\bibitem{Kra}
M.A.~Krasnosel'ski\u{i},  \emph{On a theorem of M. Riesz},
 Soviet Math. Dokl. 1 (1960), 229--231; translated from Dokl. Akad. Nauk SSSR 131 (1960), 246--248.






\bibitem{MitSto}
M.~Mitkovski, C.B.~Stockdale, \emph{On the $T1$ theorem for compactness of Calder\'on-Zygmund operators},
arXiv:2309.15819 [math.CA].



\bibitem{Molinet}
L.~Molinet, {\it On ill-posedness for the one-dimensional periodic cubic Schr\"odinger equation},
Math. Res. Lett. 16 (2009), no. 1, 111--120.


\bibitem{OS}
T.~Oh, C.~Sulem,
{\it  On the one-dimensional cubic nonlinear Schr\"odinger equation below $L^2$}, Kyoto J. Math. 52 (2012), no.1, 99--115. 


\bibitem{RS}
M.~Reed, B.~Simon, \emph{Methods of modern mathematical physics. I. Functional analysis.} Second edition. Academic Press, Inc. [Harcourt Brace Jovanovich, Publishers], New York, 1980. xv+400 pp.

\bibitem{Rudin}
W.~Rudin, \emph{Functional analysis}. Second edition. International Series in Pure and Applied Mathematics. McGraw-Hill, Inc., New York, 1991. xviii+424 pp

\bibitem{Stein}
E.~Stein, \emph{Harmonic analysis: real-variable methods, orthogonality, and oscillatory integrals.} Princeton
Mathematical Series, 43. Monographs in Harmonic Analysis, III. Princeton University Press, Princeton, NJ, 1993. xiv+695 pp.



\bibitem{Taylor}
M.E.~Taylor, {\it Partial differential equations I. Basic theory}. Second edition. Applied Mathematical Sciences, 115. Springer, New York, 2011. xxii+654 pp.

\bibitem{Uch}
 A.~Uchiyama, \emph{On the compactness of operators of Hankel type},
  Tohoku Math. J.  30 (1978), no. 1, 163--171.

\bibitem{V}
P.~Villarroya, \emph{A characterization of compactness for singular integrals},
J. Math. Pures Appl.  104 (2015), no. 3, 485--532.



\bibitem{Wil}
M.~Wilson, \emph{Convergence and stability of the Calder\'on reproducing formula in $H^1$ and $\BMO$},
J. Fourier Anal. Appl. 17 (2011), no. 5, 801--820.



\end{thebibliography}
\end{document}